\documentclass[11pt]{amsart}
\usepackage{amsmath,amssymb,amsfonts}
\usepackage[mathscr]{eucal}
\usepackage[all,2cell,color]{xy}

\usepackage{enumerate}

\usepackage[margin=1.4in]{geometry}
\usepackage{xcolor}

\usepackage{tikz-cd}
\usepackage{tikz}
\usetikzlibrary{arrows, decorations.markings, decorations.pathmorphing,shapes}
\pgfarrowsdeclarecombine{twotriang}{twotriang}{stealth}{stealth}%
{stealth}{stealth}
\tikzset{arrow/.style={-stealth}}
\tikzset{arrowshorter/.style={-stealth, shorten <=2pt, shorten >=2pt}}
\tikzset{arrowmuchshorter/.style={-stealth, shorten <=7pt, shorten >=6pt}}
\tikzset{dot/.style={circle,draw,fill,inner sep=1pt}}
\tikzset{mono/.style={>-stealth}} 
\tikzset{epi/.style={-twotriang}} 

\tikzset{twoarrowlonger/.style={double,double distance=1.5pt,
		shorten <=5pt,shorten >=6pt,
		decoration={markings,mark=at position -4pt with {\arrow[scale=1.75]{>}}},
		preaction={decorate}}}

\tikzset{vnode/.style={circle, radius=2pt, minimum size=4pt, draw, fill, inner sep=0, label={[below,text height=5mm]:#1}}}
\tikzset{boxy/.style={baseline={([yshift=0.5ex]current bounding box.center)}}}

\usetikzlibrary{shapes}
\tikzset{vellipsefirstone/.style={draw, ellipse, minimum width=0.6*#1, minimum height=1.5*#1, blue, thick}}

\tikzset{vellipsesecondone/.style={draw, ellipse, minimum width=0.6*#1, minimum height=1.5*#1, orange, thick}}

\tikzset{vellipsethirdone/.style={draw, ellipse, minimum width=0.6*#1, minimum height=1.5*#1, green!40!black!60, thick}}

\tikzset{vellipsefirsttwo/.style={draw, ellipse, minimum width=2*#1, minimum height=1.5*#1, blue, thick}}

\tikzset{vellipsesecondtwo/.style={draw, ellipse, minimum width=2*#1, minimum height=1.5*#1, orange, thick}}

\tikzset{vellipsefirstthree/.style={draw, ellipse, minimum width=3.5*#1, minimum height=1.5*#1, blue, thick}}

\theoremstyle{plain}
\newtheorem{theorem}[equation]{Theorem}

\newtheorem{prop}[equation]{Proposition}

\theoremstyle{definition}
\newtheorem{definition}[equation]{Definition}
\newtheorem{example}[equation]{Example}

\numberwithin{equation}{section}

\input xy
\xyoption{all}

\newcommand{\Cat}{\mathcal Cat}
\newcommand{\Deltaop}{\Delta^{\op}}
\newcommand{\Hom}{\operatorname{Hom}}
\newcommand{\Hor}{\mathcal Hor}
\newcommand{\hor}{\operatorname{hor}}
\newcommand{\id}{\operatorname{id}}
\newcommand{\mor}{\operatorname{mor}}
\newcommand{\nerve}{\operatorname{nerve}}
\newcommand{\ob}{\operatorname{ob}}
\newcommand{\op}{\operatorname{op}}
\newcommand{\sdc}{\mathcal{SDC}}
\newcommand{\Seg}{\mathcal Seg}
\newcommand{\Set}{\mathcal Set}
\newcommand{\Sq}{\mathcal Sq}
\newcommand{\sq}{\operatorname{sq}}

\newcommand{\Ver}{\mathcal Ver}
\newcommand{\ver}{\operatorname{ver}}

\begin{document}
	
\title{Combinatorial examples and applications of 2-Segal sets}
	
\author[J.E. Bergner]{Julia E. Bergner}
	
\address{Department of Mathematics, University of Virginia, Charlottesville, VA 22904}
	
\email{jeb2md@virginia.edu}
	
\date{\today}

\begin{abstract}
	We give an introduction to the theory of 2-Segal sets, and two of the main applications of them: Hall algebras and a discrete version of Waldhausen's $S_\bullet$-construction.  We present several combinatorial examples and how these constructions can be applied to them.
\end{abstract}
	
\maketitle

\section{Introduction}

The notion of a 2-Segal space was introduced by Dyckerhoff and Kapranov \cite{dk} and independently by G\'alvez-Carrillo, Kock, and Tonks under the name of decomposition space \cite{gckt}.  Its definition generalizes that of a Segal space, and encodes an algebraic structure like that of a category, but for which composition need not exist or be unique, but yet still satisfies associativity.  A good deal of the motivation for the definition arises from the desire to unify and generalize constructions such as Hall algebras and incidence coalgebras, and they have been shown to have a close relationship with algebraic $K$-theory.  Since their introduction, they have been shown to appear in a wide range of applications, from symplectic geometry to group theory.

Much of the literature in the subject is written in a homotopical context, allowing for robust examples and the use of sophisticated tools.  However, the limitation of this generality is that it can be intimidating to a newcomer to the subject, especially one less fluent in the language of modern homotopy theory and higher category theory.  In this paper, we give an introduction in the discrete setting, namely that of 2-Segal sets, focusing on concrete combinatorial examples and how we can apply some of the main constructions in the subject to them.

We begin with a review of simplicial sets and categories, and then use the notion of 1-Segal set to generalize to the definition of 2-Segal set.  We then develop two families of examples: the first a notion of 2-Segal sets associated to graphs, and the second a similar construction for trees.  Finally, we describe two of the main constructions for 2-Segal sets: the Hall algebra and a discrete version of the Waldhausen $S_\bullet$-construction.

For further expository material on 2-Segal spaces, we refer the reader to \cite{hackney}, and for the connection with Hall algebras to \cite{cy}.

These notes are an expansion of a talk at the ECOGyT conference in Bogot\'a in 2024.  We would like to thank the organizers of that conference both for the invitation to speak, and for the opportunity to develop this exposition.  Many of the diagrams in this paper are taken from the papers \cite{waldsets}, \cite{trees}, and \cite{simpsets}; we would like to thank the coauthors on the first two papers for creating them, and Walker Stern for his assistance on the ones taken from the third listed paper.

\section{Simplicial sets, categories, and Segal sets}

Although we are keeping things more simple in this paper by avoiding more complicated simplicial objects, it is necessary to work in the context of simplicial sets.  Here, we give a brief review of them, and refer the reader to \cite{simpsets} for a more detailed and conceptual development of the theory; see also \cite{friedman}.

The category $\Delta$ has objects the finite ordered sets $[n]=\{0 \leq 1 \leq \cdots \leq n\}$ and morphisms the order-preserving functions.  We can depict this category as follows:
\[ \xymatrix@1{[0] \ar@<.5ex>[r] \ar[r] & [1] \ar@<.5ex>[l] \ar[r] \ar@<.5ex>[r] \ar@<1ex>[r] & [2] \ar@<.5ex>[l] \ar@<1ex>[l] \ar[r] \ar@<.5ex>[r] \ar@<1ex>[r] \ar@<1.5ex>[r] & \cdots. \ar@<.5ex>[l] \ar@<1ex>[l] \ar@<1.5ex>[l]} \] 

For each $0 \leq i \leq n$, we denote by $d^i \colon [n] \rightarrow [n+1]$ the \emph{coface map} given by 
\[ d^i(j) = \begin{cases}
	j & j< i \\
	j+1 & j \geq i,
	\end{cases} \]
and by $s^i \colon [n] \rightarrow [n-1]$ the \emph{codegeneracy map} given by 
\[ s^i(j) = \begin{cases}
	j & j \leq i \\
	j-1 & j >i.
	\end{cases} \]
Observe that these maps, subject to various cosimplicial relations, generate all the morphisms in $\Delta$.

Given a small category $\mathcal C$, recall that its \emph{opposite category} $\mathcal C^{\op}$ is the category with the same objects as $\mathcal C$ but with the direction of its morphisms reversed, so that for any objects $a$ and $b$ of $\mathcal C$, we have $\Hom_{\mathcal C^{\op}}(a,b) = \Hom_\mathcal C(b,a)$.

We can now define simplicial sets.

\begin{definition}
	A \emph{simplicial set} is a functor $K \colon \Deltaop \rightarrow \Set$, where $\Set$ denotes the category of sets and functions.
\end{definition}

We can thus depict a simplicial set $K$ as a diagram
\[ \xymatrix@1{K_0 \ar[r] & K_1 \ar@<.8ex>[l] \ar@<-.8ex>[l]  \ar@<-.5ex>[r] \ar@<.5ex>[r] & K_2 \cdots. \ar[l] \ar@<1ex>[l] \ar@<-1ex>[l]} \]
Now, the coface maps in $\Delta$ induce the \emph{face maps} $d_i \colon K_n \rightarrow K_{n-1}$, and the codegeneracy maps induce the \emph{degeneracy maps} $s_i \colon K_n \rightarrow K_{n+1}$.

Given a simplicial set $K$, and an object $[n]$ of $\Delta$, we usually denote the set $K([n])$ by $K_n$.   We refer to the objects of this set as $n$-\emph{simplices} of $K$, a name further justified by the following example.

\begin{example}
	For $n \geq 0$, we define the $n$-\emph{simplex} $\Delta[n]$ is the simplicial set defined by
	\[ \Delta[n]_k = \Hom_\Delta([k],[n]). \]
	In other words, $\Delta[n]$ is the representable functor on the object $[n]$.  As a consequence, Yoneda's Lemma tells us that, for any simplicial set $K$, there is a natural isomorphism
	\[ \Hom(\Delta[n], K) \cong K_n. \]
\end{example}

\begin{example}
	We usually define a \emph{monoid} $M$ to be a set equipped with a binary operation that has an identity element and is associative. 
	
	Given a monoid in this sense, we can associate to it a simplicial set:
	\[ \xymatrix@1{\{e\} \ar[r] & M \ar@<.8ex>[l] \ar@<-.8ex>[l]  \ar@<-.5ex>[r] \ar@<.5ex>[r] & M \times M \cdots. \ar[l] \ar@<1ex>[l] \ar@<-1ex>[l]} \]
	
	A more general small category $\mathcal C$ can simarly be thought of a simplicial set via its \emph{nerve}:
	\[ \xymatrix@1{\ob(\mathcal C) \ar[r] & \mor(\mathcal C) \ar@<.8ex>[l] \ar@<-.8ex>[l]  \ar@<-.5ex>[r] \ar@<.5ex>[r] & \mor(\mathcal C) \times_{\ob(\mathcal C)} \mor(\mathcal C) \cdots. \ar[l] \ar@<1ex>[l] \ar@<-1ex>[l]} \]
\end{example}

More formally, we can make the following definition.

\begin{definition}
	Let $\mathcal C$ be a small category.  Its \emph{nerve}, denoted by $\nerve(\mathcal C)$ is the simplicial set defined by
	\[ \nerve(\mathcal C)_n = \Hom_{\Cat}([n], \mathcal C). \]
\end{definition}

\begin{example}
	For any $n \geq 0$, we can think of the ordered set $[n]$ as a category with objects the elements $0, 1, \ldots, n$ and a single morphism $i \rightarrow j$ whenever $i \leq j$.  Then the nerve of $[n]$ is precisely the $n$-simplex $\Delta[n]$.
\end{example}

A natural question is the following: how can we identify the simplicial sets that can be obtained as nerves of categories?  To answer this question, we make use of another example of a simplicial set.

\begin{example}
	For any $n \geq 0$, let $G(n)$ be the simplicial set 
	\[ \Delta[1] \amalg_{\Delta[0]} \cdots \amalg_{\Delta[0]} \Delta[1], \]
	where the pushouts are indicated so that the ``target" of each 1-simplex is the ``source" of the next; in other words, $G(n)$ is the colimit of the diagram
	\[ \Delta[1] \xleftarrow{d^0} \Delta[0] \xrightarrow{d^1} \xleftarrow{d^0} \cdots \xrightarrow{d^1} \Delta[1]. \]
	For example, when $n=2$, we can depict $G(2)$ as the spine
	\[ \begin{tikzpicture}
		\foreach \tet/\lab/\nom/\pos in {-30/A/2/below right,90/B/1/above,210/C/0/below left}{
			\draw[fill=black] (\tet:1.5) circle (0.05);
			\path (\tet:1.5) node[label=\pos:$\nom$] (\lab) {};
		}
		
		\draw[->] (C) to (B);
		\draw[->] (B) to (A);
		%\draw[->] (C) to (A);
	\end{tikzpicture} \]
	of the 2-simplex
	\[ \begin{tikzpicture}
		\foreach \tet/\lab/\nom/\pos in {-30/A/2./below right,90/B/1/above,210/C/0/below left}{
			\draw[fill=black] (\tet:1.5) circle (0.05);
			\path (\tet:1.5) node[label=\pos:$\nom$] (\lab) {};
		}
		\path[fill=blue, opacity=0.3] (A.center) -- (B.center) -- (C.center) -- cycle;
		\draw[->] (C) to (B);
		\draw[->] (B) to (A);
		\draw[->] (C) to (A);
	\end{tikzpicture} \]  
\end{example}

Given any simplicial set $K$, the inclusion maps $G(n) \rightarrow \Delta[n]$ induce maps of sets
\[ K_n \cong \Hom(\Delta[n], K) \rightarrow \Hom(G(n),K) \cong \underbrace{K_1 \times_{K_0} \cdots \times_{K_0} K_1}_n. \]
We call these maps the \emph{1-Segal maps}.  The ``1" here is meant to emphasize that the $G(n)$'s are 1-dimensional sub-simplicial sets, in that they have no nondegenerate simplices above dimension 1.  Elsewhere in the literature they are often simply called \emph{Segal maps}.

\begin{definition}
	A \emph{1-Segal set} is a simplicial set $K$ for which the Segal maps are isomorphisms for all $n \geq 2$.
\end{definition}

Note that we omit the cases $n=0,1$ here because those maps are by definition the identity, and therefore automatically isomorphisms. 

Observe that the 1-Segal map for $n=2$ is an isomorphism precisely when any pair of ``composable" 1-simplices has a unique composite, given by the remaining edge in the boundary of the 2-simplex.  In fact, we have the following result.

\begin{prop}
	A simplicial set $K$ is isomorphic to the nerve of a category if and only if it is a 1-Segal set.
\end{prop}

\begin{proof}
	It is not hard to check that the nerve of a category is a 1-Segal set.  Conversely, let $K$ be a 1-Segal set.  Define a category with object set $K_0$ and morphisms $K_1$, with source and target maps given by $d_1$ and $d_0$, respectively, and identity maps given by degenerate 1-simplices.  To define composition, consider the diagram
	\begin{equation} \label{comp}
		K_1 \times_{K_0} K_1 \xleftarrow{(d_0,d_2)} K_2 \xrightarrow{d_1} K_1. 
	\end{equation}
	By the 1-Segal condition, the first map is an isomorphism, so we can define a composition law via the composite $d_1(d_0,d_2)^{-1}$.  We leave the verification of associativity and unitality to the reader.
\end{proof}

\section{2-Segal sets} 

In this section, we introduce 2-Segal sets and their algebraic structure.

If we think of $G(n)$ as a ``triangulation" of a line segment, then we can generalize the Segal maps by moving up a dimension and triangulating polygons.  In the first nontrivial case, there are two triangulations of a square:
\begin{equation} \label{squares}
	\xymatrix{\mathcal T_1: & 3 & 2 \ar[l]  & \mathcal T_2:& 3 & 2 \ar[l] \\
	& 0 \ar[u] \ar[r] \ar[ur] & 1 \ar[u] && 0 \ar[u] \ar[r] & 1. \ar[u] \ar[ul]} 
\end{equation}
Each gives two faces of the boundary of a 3-simplex $\Delta[3]$.  Any triangulation of an $(n+1)$-gon similarly induces a map
\[ K_n \rightarrow \underbrace{K_2 \times_{K_1} \cdots \times_{K_1} K_2.}_{n-1} \]

As we did before, we can map these inclusions of simplicial sets into a fixed simplicial set $K$:
\[ \xymatrix{& \Hom(\mathcal T_1, K) \cong K_2 \times_{K_1} K_2 \\
	K_3 \cong \Hom(\Delta[3], K) \ar[ur]^{(d_1, d_3)} \ar[dr]^{(d_0, d_2)} & \\
	& \Hom(\mathcal T_2, K) \cong K_2 \times_{K_1} K_2.} \]

\begin{definition}
	A \emph{2-Segal set} is a simplicial set $K$ such that, for every $n \geq 3$ and any triangulation $\mathcal T$ of a regular $(n+1)$-gon with cyclically labeled vertices, the associated 2-Segal map
	\[ K_n \rightarrow \underbrace{K_2 \times_{K_1} \cdots \times_{K_1} K_2}_{n-1} \]
	is an isomorphism.
\end{definition}

Just as 1-Segal set has an algebraic description, as the nerve of a category, a 2-Segal set can be considered as a similar, but weaker, algebraic object.  As before, given a 2-Segal set $K$, we can take the sets $K_0$ and $K_1$ as the sets of objects, and morphisms, respectively.  But now, if we try to define a composition law as in \eqref{comp}, the left-going map need not be an isomorphism.  The only thing we can do, given an element $(x,y) \in K_1 \times_{K_0} K_1$, is take its preimage under the map $(d_0, d_2)$ and apply the map $d_1$.  Thus, the composite of $(x,y)$ need not exist if $(d_0, d_2)$ is not surjective, and need not be unique if this map is not injective.  

However, a close inspection of the triangulations in \eqref{squares} shows that this composition is associative.  The boundary of each square can be thought of as three composable arrows, together with a composite of all three.  The fact that each of the two triangulations uniquely determines the same 3-simplex establishes that this composite is independent of the possible bracketings for the three arrows.

\begin{example}
	A \emph{partial monoid} is a set $M$ equipped with a domain of multiplication $M_2 \subseteq M \times M$ together with a binary operation 
	\[ \cdotp \colon M_2 \rightarrow M \]
	such that:
	\begin{itemize}
		\item $(m \cdotp m') \cdotp m''$ is defined if and only if $m \cdotp (m' \cdotp m'')$ is defined and they are equal; and
		
		\item there is an identity element $1 \in M$ such that $(1,m), (m,1) \in M_2$ and $1 \cdotp m= m = m \cdotp 1$ for all $m \in M$.
	\end{itemize}

	As for a ordinary monoid, we can take a nerve construction, as follows.  Let $M_0={1}$.  For $k \geq 1$, let 
	\[ M_k \subseteq M^k \]
	be the subset of composable k-tuples, or elements 
	\[ (m_1, \ldots, m_k) \in M^k \]
	such that 
	\[ (m_1 \cdots m_i, m_{i+1}) \in M_2 \]
	for every $1 \leq i <k$.  Face maps are given by composition, and the degeneracy maps are defined by inserting the identity element 1 in the indicated position. 

	The nerve of a partial monoid is a 2-Segal set that is not 1-Segal, since composition does not always exist.  Observe that when it does exist, composition is unique, however.
\end{example}

\section{2-Segal sets from graphs}

Let us now focus on a family of combinatorial examples arising from finite graphs, as described in \cite{waldsets}.  It was inspired by the decomposition space of all graphs as described in \cite{gckt}, and can be thought of as providing discrete versions of the construction there, corresponding to individual graphs.

\begin{definition}
Let $G$ be a finite graph.  We associate to $G$ a simplicial set $X^G$ as follows.
\begin{itemize}
	\item The set $X^G_0$ has a single element, denoted by $\varnothing$.
	
	\item The set $X^G_1$ is the set of all subgraphs of $G$.  The empty subgraph $\varnothing$ is the single degenerate 1-simplex.
	
	\item Any $X^G_n$ has elements $(H; S_1, \ldots, S_n)$ where $H$ is a subgraph of $G$ and the sets $S_1, \ldots, S_n$ form a partition of the set of vertices into $n$ disjoint (but possibly empty) sets. 
\end{itemize}
\end{definition}

\begin{example}
Consider the graph 
\begin{center}
	\begin{tikzpicture}
		\def\l{1cm} %length for the edges of the graph
		%the graph itself
		\draw (0,0) node (G){$G:$};
		\draw (\l,0)  node[vnode=$a$](a){};
		\draw (2*\l,0) node[vnode=$b$](b){};
		\draw (3*\l,0) node[vnode=$c.$](c){};
		\draw (a)--(b);
		\draw (b)--(c);
	\end{tikzpicture}
\end{center}
By definition, we have $X^G_0 = \{\varnothing\}$.  The 1-simplices are given by: 

\begin{tikzpicture}
	\def\l{.9cm} 
	
	%1-simplices
	\draw (0,0) node {$X^G_1:$};
	\draw (\l,0) node {$\varnothing$};
	\draw (\l,0) node[vellipsefirstone=\l]{};
	
	\begin{scope}[xshift=2*\l]
		%length for the edges of the graph
		\draw (0,0) node [vnode=$a$](a){};
		\draw (0,0) node[vellipsefirstone=\l]{};
	\end{scope}
	
	\begin{scope}[xshift=3*\l]
		\draw (0,0) node [vnode=$b$](b){};
		\draw (0,0) node[vellipsefirstone=\l]{};
	\end{scope}
	
	\begin{scope}[xshift=4*\l]
		\draw (0,0) node [vnode=$c$](c){};
		\draw (0,0) node[vellipsefirstone=\l]{};
	\end{scope}
	
	\begin{scope}[xshift=4*\l]
		\draw (\l,0) node [vnode=$a$](a){};
		\draw (2*\l,0) node [vnode=$b$](b){};
		\draw (1.5*\l,0) node[vellipsefirsttwo=\l]{};
	\end{scope}
	
	\begin{scope}[xshift=6.5*\l]
		\draw (\l,0) node [vnode=$a$](a){};
		\draw (2*\l,0) node [vnode=$b$](b){};
		\draw (a)--(b);
		\draw (1.5*\l,0) node[vellipsefirsttwo=\l]{};
	\end{scope}
	
	\begin{scope}[xshift=9.0*\l]
		\draw (\l,0) node [vnode=$a$](a){};
		\draw (2*\l,0) node [vnode=$c$](c){};
		\draw (1.5*\l,0) node[vellipsefirsttwo=\l]{};
	\end{scope}
	
	\begin{scope}[yshift=-2*\l,xshift=\l]
		\draw (\l,0) node [vnode=$b$](b){};
		\draw (2*\l,0) node [vnode=$c$](c){};
		\draw (1.5*\l,0) node[vellipsefirsttwo=\l]{};
	\end{scope}
	
	\begin{scope}[yshift=-2*\l, xshift=3.5*\l]
		\draw (\l,0) node [vnode=$b$](b){};
		\draw (2*\l,0) node [vnode=$c$](c){};
		\draw (b)--(c);
		\draw (1.5*\l,0) node[vellipsefirsttwo=\l]{};
	\end{scope}
	
	\begin{scope}[yshift=-2*\l, xshift=6*\l]
		\draw[fill] (\l,0) node [vnode=$a$](a){};
		\draw (2*\l,0) node [vnode=$b$](b){};
		\draw (3*\l,0) node [vnode=$c$](c){};
		\draw (2*\l,0) node[vellipsefirstthree=\l]{};
	\end{scope}
	
	\begin{scope}[yshift=-4*\l,xshift=\l]
		\draw (\l,0) node [vnode=$a$](a){};
		\draw (2*\l,0) node [vnode=$b$](b){};
		\draw (3*\l,0) node [vnode=$c$](c){};
		\draw (a)--(b);
		\draw (2*\l,0) node[vellipsefirstthree=\l]{};
	\end{scope}
	
	\begin{scope}[yshift=-4*\l, xshift=5.5*\l]
		\draw (\l,0) node [vnode=$a$](a){};
		\draw (2*\l,0) node [vnode=$b$](b){};
		\draw (3*\l,0) node [vnode=$c$](c){};
		\draw (b)--(c);
		\draw (2*\l,0) node[vellipsefirstthree=\l]{};
	\end{scope}
	
	\begin{scope}[yshift=-6*\l,xshift=3*\l]
		\draw (\l,0) node [vnode=$a$](a){};
		\draw (2*\l,0) node [vnode=$b$](b){};
		\draw (3*\l,0) node [vnode=$c.$](c){};
		\draw (a)--(b);
		\draw (b)--(c);
		\draw (2*\l,0) node[vellipsefirstthree=\l]{};
	\end{scope}
\end{tikzpicture} 
\end{example}

Rather than write out a list of 2-simplices, of which there are many, let us illustrate the simplicial structure using them.  In our depiction of 2-simplices, we assume that the sets in blue come before the sets in orange in the ordering in the partition.  Face maps of 2-simplices are given by cutting or merging:
\begin{center}
	\begin{tikzpicture}
		\def\l{1cm} %length for the edges of the graph
		%the original 2-simplex
		\begin{scope}
			
			%the subgraph 
			\draw[fill] (\l,0) node [vnode=$a$](a){};
			\draw (2*\l,0) node [vnode=$b$](b){};
			\draw (3*\l,0) node [vnode=$c$](c){};
			\draw (a)--(b);
			
			%partitioning
			\draw (\l,0) node[vellipsefirstone=\l]{};
			\draw (2.5*\l, 0) node[vellipsesecondtwo=\l]{};
		\end{scope}
		
		%d0
		\begin{scope}[yshift={-1.7*\l}, xshift={-3*\l}]
			\draw (\l,0) node [vnode=$b$](b){};
			\draw (2*\l,0) node [vnode=$c$](c){};
			\draw (1.5*\l, 0) node[vellipsefirsttwo=\l]{};
		\end{scope}
		
		%d1
		\begin{scope}[yshift=-2.2*\l]
			\draw (\l,0) node [vnode=$a$](a){};
			\draw (2*\l,0) node [vnode=$b$](b){};
			\draw (3*\l,0) node [vnode=$c$](c){};
			\draw (a)--(b);
			\draw (2*\l, 0) node[vellipsefirstthree=\l]{};
		\end{scope}
		
		%d2
		\begin{scope}[yshift=-1.7*\l, xshift=4*\l]
			\draw (\l,0) node [vnode=$a.$](a){};
			\draw (\l,0) node[vellipsefirstone=\l]{};
		\end{scope}
		
		%arrows
		\draw[|->, thick] (0.5*\l, -0.1*\l)--node[above](d0){$d_0$}(-1.0*\l, -1.0*\l);
		\draw[|->, thick] (2*\l, -0.8*\l)--node[left](d1){$d_1$}(2*\l, -1.3*\l);
		\draw[|->, thick] (3.5*\l, -0.6*\l)--node[above](d2){$d_2$}(4.4*\l, -1.3*\l);
	\end{tikzpicture}
\end{center}
On the other hand, degeneracy maps are given by inserting empty sets into the partition, with the location specified by the indexing:
\begin{center}
	\begin{tikzpicture}
		\def\l{1cm} %length for the edges of the graph
		%the original 1-simplex
		\begin{scope}
			
			%the subgraph 
			\draw (\l,0) node [vnode=$a$](a){};
			\draw (2*\l,0) node [vnode=$b$](b){};
			\draw (a)--(b);
			\draw (1.5*\l,0) node[vellipsefirsttwo=\l]{};
		\end{scope}
		
		%s0
		\begin{scope}[yshift={-2*\l}, xshift={-2*\l}]
			\draw (\l,0) node (emp){$\varnothing$};
			\draw (2*\l,0) node [vnode=$a$](a){};
			\draw (3*\l,0) node [vnode=$b$](b){};
			\draw (a)--(b);
			%partitioning
			\draw (\l,0) node[vellipsefirstone=\l]{};
			\draw (2.5*\l, 0) node[vellipsesecondtwo=\l]{};
		\end{scope}
		
		%s1
		\begin{scope}[yshift=-2*\l, xshift=2*\l]
			\draw[fill] (\l,0) node [vnode=$a$](a){};
			\draw (2*\l,0) node [vnode=$b$](b){};
			\draw (3*\l,0) node(emp){$\varnothing.$};
			\draw (a)--(b);
			\draw (1.5*\l,0) node[vellipsefirsttwo=\l]{};
			\draw (3*\l, 0) node[vellipsesecondone=\l]{};
		\end{scope}
		
		%arrows
		\draw[|->, thick] (0.4*\l, -0.2*\l)--node[above, yshift=0.1cm](s0){$s_0$}(-0.2*\l, -1.0*\l);
		\draw[|->, thick] (2.6*\l, -0.2*\l)--node[above](s1){$s_1$}(3.5*\l, -1.0*\l);
	\end{tikzpicture}
\end{center}

The following diagram illustrates that $X^G$ is not 1-Segal, since a composite of the subgraphs $a$ and $b$ can be obtained from each of the two indicated 2-simplices:
\begin{center}
	\begin{tikzpicture}
		\def\l{1cm} %length for the edges of the graph
		\begin{scope}
			\draw (\l,0) node [vnode=$a$](a){};
			\draw (2*\l,0) node [vnode=$b$](b){};
			\draw (\l,0) node[vellipsefirstone=\l]{};
			\draw (2*\l, 0) node[vellipsesecondone=\l]{};
			%faces
			\draw (-0.3*\l,-0.7*\l) node [vnode=$a$](a2){};
			\draw (3.3*\l,-0.7*\l) node [vnode=$b$](b2){};
			%partitions for faces
			\draw (-0.3*\l, -0.7*\l) node[vellipsefirstone=\l]{};
			\draw (3.3*\l,-0.7*\l) node[vellipsefirstone=\l]{};
			%arrows
			\draw[|->, shorten >=0.35cm, shorten <=0.35cm, thick] (a)--node[above, xshift=-0.1cm]{$d_2$}(a2);
			\draw[|->, shorten >=0.35cm, shorten <=0.35cm, thick] (b)--node[above, xshift=0.1cm]{$d_0$}(b2);
		\end{scope}
		
		\draw (4.2*\l, 0) node{$\quad$and$\quad$};
		
		\begin{scope}[xshift=5.5*\l]
			\draw (\l,0) node [vnode=$a$](a){};
			\draw (2*\l,0) node [vnode=$b$](b){};
			\draw (a)--(b);
			\draw (\l,0) node[vellipsefirstone=\l]{};
			\draw (2*\l, 0) node[vellipsesecondone=\l]{};
			%faces
			\draw (-0.3*\l,-0.7*\l) node [vnode=$a$](a2){};
			\draw (3.3*\l,-0.7*\l) node [vnode=$b.$](b2){};
			%partitions for faces
			\draw (-0.3*\l, -0.7*\l) node[vellipsefirstone=\l]{};
			\draw (3.3*\l,-0.7*\l) node[vellipsefirstone=\l]{};
			%arrows
			\draw[|->, shorten >=0.35cm, shorten <=0.35cm, thick] (a)--node[above, xshift=-0.1cm]{$d_2$}(a2);
			\draw[|->, shorten >=0.35cm, shorten <=0.35cm, thick] (b)--node[above, xshift=0.1cm]{$d_0$}(b2);
		\end{scope}
	\end{tikzpicture}
\end{center}
This example shows that the 1-Segal map $(d_0,d_2) \colon X_2 \rightarrow X_1 \times_{X_0} X_1$ need not be injective. Furthermore, 1-simplices given by non-disjoint subgraphs cannot be composed, since overlapping subgraphs cannot be obtained as $d_0$ and $d_2$ of the same 2-simplex.  Thus, this 1-Segal map also need not be surjective.  However, we do have the following result.

\begin{theorem} \cite[\S 2]{waldsets}
	For any graph $G$, the simplicial set $X^G$ is 2-Segal.
\end{theorem}

The idea behind the proof is illustrated in the following diagram:
\begin{center}
	\begin{tikzpicture}
		\def\l{1cm} %length for the edges of the graph
		%the original 2-simplex
		\begin{scope}
			
			%the subgraph 
			\draw[fill] (\l,0) node [vnode=$a$](a){};
			\draw (2*\l,0) node [vnode=$b$](b){};
			\draw (3*\l,0) node [vnode=$c$](c){};
			\draw (b)--(c);
			
			%partitioning
			\draw (\l,0) node[vellipsefirstone=\l]{};
			\draw (2*\l, 0) node[vellipsesecondone=\l]{};
			\draw (3*\l, 0) node[vellipsethirdone=\l]{};
		\end{scope}
		
		%d0
		\begin{scope}[yshift=-3*\l, xshift=\l]
			\draw (\l,0) node [vnode=$b$](b){};
			\draw (2*\l,0) node [vnode=$c$](c){};
			\draw (b)--(c);
			\draw (\l,0) node[vellipsefirstone=\l]{};
			\draw (2*\l, 0) node[vellipsesecondone=\l]{};
		\end{scope}
		
		%d2
		\begin{scope}[ xshift=4*\l]
			\draw[fill] (\l,0) node [vnode=$a$](a){};
			\draw (2*\l,0) node [vnode=$b$](b){};
			\draw (3*\l,0) node [vnode=$c$](c){};
			\draw (b)--(c);
			\draw (\l,0) node[vellipsefirstone=\l]{};
			\draw(2.5*\l, 0) node[vellipsesecondtwo=\l]{};
		\end{scope}
		
		\begin{scope}[yshift=-3*\l, xshift=5*\l]
			\draw (\l,0) node [vnode=$b$](b){};
			\draw (2*\l,0) node [vnode=$c$](c){};
			\draw (b)--(c);
			\draw (1.5*\l, 0) node[vellipsefirsttwo=\l]{};
		\end{scope}
		
		%arrows
		\draw[|->, thick] (2*\l, -1*\l)--node[left](d0){$d_0$}(2*\l, -2.1*\l);
		\draw[|->, thick] (3.5*\l, 0)--node[above](d2){$d_2$}(4.5*\l, 0);
		\draw[|->, thick] (3.5*\l, -3*\l)--node[below](d1){$d_1$}(5.4*\l, -3*\l);
		\draw[|->, thick] (6*\l, -1*\l)--node[right](d0r){$d_0$}(6*\l, -2.2*\l);
	\end{tikzpicture}
\end{center}
Observe that we could recover the 3-simplex in the top left-hand corner from the data of the rest of the diagram: the 2 -simplex in the top right-hand corner gives us the underlying graph, and the 2-simplex in the lower left-hand corner shows us how to refine the partition to get the desired 3-simplex.  Here, note that we regard the green set as coming after the orange one in the ordering on the partition.

\section{2-Segal sets from trees}

Again, inspired by a decomposition space example from \cite{gckt}, in \cite{trees} we developed examples of 2-Segal sets coming from rooted trees.  Here, we think of a \emph{tree} as a graph with no cycles, with one vertex specified as its root.  As with the graphs above, we assume that the vertices of a tree are labeled.  We begin with some necessary definitions.

\begin{definition} 
	Let $T$ be a rooted tree. An \emph{admissible cut} on $T$ is a partition of the vertices $V(T)= L \amalg U$ such that $L$ is a \emph{lower subtree} of $T$, in that it is either a tree containing the root of $T$ or empty.
\end{definition}

The following definition provides a way to have multiple admissible cuts on the same tree.

\begin{definition}
	A \emph{layering of $n-1$ cuts} of a rooted tree $T$ is a sequence of partially ordered subsets of vertices of $T$
	\[V(T)=L_0\supseteq L_1 \supseteq \cdots\supseteq L_n=\varnothing,\] such that each $L_k$ is a lower subtree of $T$. 
\end{definition}

Similarly to graphs equipped with an ordered partition of their edges, we can think of a rooted tree $T$ with a layering of $n-1$ cut as an $n$-simplex of a simplicial set $X^T$, with face maps given by cutting and merging along cuts.  For example, consider the following tree with two cuts, and then the tree obtained by removing all vertices and edges that appear above the topmost cut:
\[ \begin{tikzpicture}[scale=.60,line cap=round,line join=round,>=triangle 45,x=1cm,y=1cm, dot/.style = {circle, gray, minimum size=#1, inner sep=0pt, outer sep=0pt}, 
		dot/.default = 3pt]
		
		\draw [-] (0,0)-- (1,1);
		\draw [] (0,0)-- (-1,1);
		\draw [] (0,-1)-- (0,0);
		\draw [] (-1,1)-- (1,3);
		\draw [] (0,2)-- (-1,3);
		\draw [] (-1,1)-- (-2,2);
		\draw [-, dashed] (-2.12,1.43)-- (1.20,2.8);
		\draw [-, dashed] (-1.05,-.5)-- (1.20,-.5);
		\draw [fill=black] (0,-1) circle (3pt);
		\draw[color=black] (-.4,-.8) node {$h$};
		%	\draw[color=black] (0,-2) node {(I)};
		\draw [fill=white] (-1,1) circle (3pt);
		\draw[color=black] (-1,1.43) node {$e$};
		\draw [fill=white] (1,1) circle (3pt);
		\draw[color=black] (1.16,1.43) node {$f$};
		\draw [fill=white] (0,0) circle (3pt);
		\draw[color=black] (0.3,-0.24) node {$g$};
		\draw [fill=white] (0,2) circle (3pt);
		\draw [fill=white] (-2,2) circle (3pt);
		\draw[color=black] (-1.90,2.43) node {$a$};
		\draw [fill=white] (-1,3) circle (3pt);
		\draw[color=black] (-0.84,3.43) node {$b$};
		\draw [fill=white] (1,3) circle (3pt);
		\draw[color=black] (1.16,3.43) node {$c$};
		\draw[color=black] (0.3,2.01) node {$d$};
	\end{tikzpicture}
	\hspace{2.5cm}
	\begin{tikzpicture}[scale=.60,line cap=round,line join=round,>=triangle 45,x=1cm,y=1cm, dot/.style = {circle, gray, minimum size=#1,
			inner sep=0pt, outer sep=0pt}, 
		dot/.default = 3pt]
		%	\clip(-5.72,-4.88) rectangle (5.72,4.88);
		\draw [] (-2,-1)-- (-2,0);
		\draw [] (-2,0)-- (-1,1);
		\draw [] (-2,0)-- (-3,1);
		\draw [] (-3,1)-- (-2,2);
		\draw [-, dashed] (-3.05,-.6)-- (-1.20,-.6);
		%\draw [] (3,0)-- (2,1);
		%	\begin{scriptsize}
			\draw [fill=white] (-2,2) circle (3pt);
			%	\draw[color=black] (-2,-2) node {(II)};
			\draw [fill=white] (-2,0) circle (3pt);
			\draw[color=black] (-1.70,-.20) node {$g$};
			\draw [fill=white] (-3,1) circle (3pt);
			\draw[color=black] (-1.7,2.1) node {$d$};
			% \draw[color=black] (-1.70,-.37) node {$d$};
			\draw[color=black] (-2.98,1.43) node {$e$};
			\draw [fill=white] (-1,1) circle (3pt);
			\draw[color=black] (-0.84,1.43) node {$f$};
			\draw [fill=black] (-2,-1) circle (3pt);
			\draw[color=black] (-2.3,-1) node {$h$.};
\end{tikzpicture} \]  
We can think of the second tree as being the lowest face of the first tree; analogously to the graph example, we can obtain other faces by removing cuts or the lowest part of the tree.  Similarly, we can define degeneracy maps by repeating cuts, or inserting them above or below the tree.

Note that taking the top face map, which discards everything below the lowest cut, can result in a \emph{forest}, or disjoint union of trees, rather than a single tree.  Thus, we need to extend our definition of admissible cuts to forests, not just to trees.  We refer the reader to \cite{trees} for a more nuanced definition, and simply suggest here that the following one can be made more precise.

\begin{definition}
	An \emph{admissible subforest} of $T$ is any forest (or tree) obtained from face maps along admissible cuts.
\end{definition}

Formalizing the idea described above, we make the following definition.

\begin{definition}
	Let $T$ be a labelled rooted tree. We define a simplicial set $X^T$ as follows:
	\begin{itemize}
		\item $X^T_0=\{\varnothing\}$;
		
		\item $X^T_1$ is the set of admissible subforests of $T$; and
		
		\item for $n \geq 2$, $X^T_n$ is the set of all admissible subforests $H\in X^T_1$ together with a layering of $n-1$ cuts $H\supseteq L_1\supseteq\dots\supseteq L_n=\varnothing$.
	\end{itemize}
\end{definition}

\begin{theorem} \cite{trees}
	For any rooted tree $T$, the simplicial set $X^T$ is 2-Segal. 
\end{theorem}

We illustrate the idea of the proof with an example, as we did for the case of a graph.  When $n=3$, the 2-Segal maps $X_3^T \to X_2^T \times_{X_1^T}X_2^T$ induced by the triangulations of \eqref{squares} need to be isomorphisms; we illustrate an example of the relevant pushouts in Figure \ref{2segfig}. 
\begin{figure}
	\centering
	\includegraphics[scale=0.7]{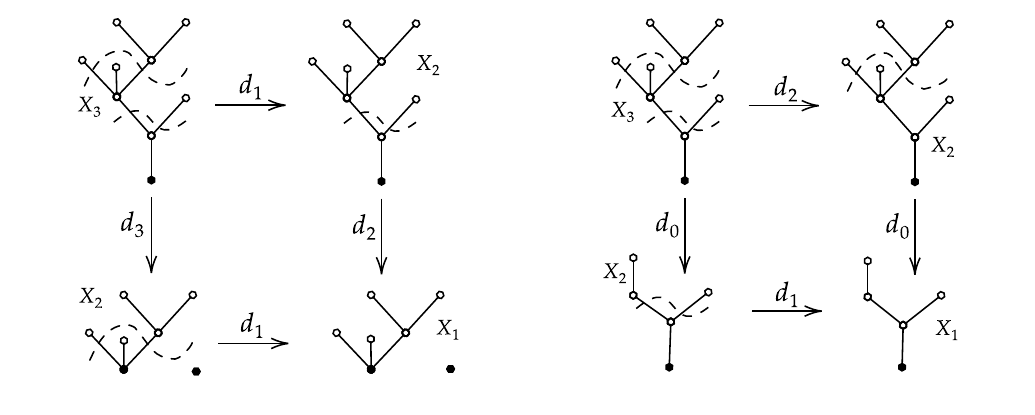}
	\caption{2-Segal maps for $X^T$. }
	\label{2segfig}
\end{figure}
To recover the trees with two cuts in the upper left corners, observe that the trees in the top right corners contain all the necessary data except one cut. The information of this missing cut is present in the bottom left corners. Their intersection, the tree in the bottom right corner, informs us how to glue these two trees, hence determines the missing cut.

The simplicial set $X^T$ is never $1$-Segal, unless $T$ is the empty tree. Indeed, if $T \neq\varnothing$ then the Segal map
\[ (d_0,d_2) \colon X_2^T\to X_1^T\times X_1^T \]
is not surjective as $(T,T)$ is not in the image.

However, unlike the case for graphs, it is always injective.

\begin{prop} \cite{trees}
	If $T$ is a labelled tree with vertex set $V$, then the Segal map
	\[ (d_0,d_2)\colon X_2^T\to X_1^T\times X_1^T \]
	is injective.  
\end{prop}

The idea is that rigidity in the definition of a cut, and inherent ordering of the vertices in a rooted tree, prevents the kind of counterexample that we had in the graph example.

Observe that if $T$ is a rooted tree with underlying graph $G$, then $X^T$ is a sub-2-Segal set of $X^G$.  

\section{Hall algebras of reduced 2-Segal sets}

One of the original motivations for developing 2-Segal structures was the unification of various Hall algebra constructions \cite{dk}.  We give the most basic version here for suitable 2-Segal sets.

\begin{definition}
	Let $X$ be a 2-Segal set with $X_0=\ast$ and only finitely many simplices in each dimension.  Define its \emph{Hall algebra} $\mathcal H(X)$ over a field $\Bbbk$ to be the $\Bbbk$-vector space generated by $X_1$.  Its multiplication is given by
	\[ x \cdotp y = \sum_z g^z_{xy}z, \]
	where $g^z_{xy}$ counts the number of elements $\sigma \in X_2$ with $d_0(\sigma)=y$, $d_1(\sigma)=z$, and $d_2(\sigma)=x$.
\end{definition}

\begin{prop} \cite[\S 3.4]{dk}
	If $X$ is as above, then its Hall algebra $\mathcal H(X)$ is an associative, unital algebra.
\end{prop}

Let us compute this algebra for examples arising from graphs.

\begin{example}
	Let $G$ be the graph
	\[ \xymatrix@1{\bullet_a \ar@{-}[r] & \bullet_b.} \]
	Then $\mathcal H(X^G)$ is 5-dimensional, spanned by $\varnothing$, $a$, $b$, $a \amalg b$, and $G$.  The only nontrivial multiplication is 
	\[ a \cdotp b = (a \amalg b) + G = b \cdotp a. \]
\end{example}

\begin{example}
	Let $G$ be the graph with a single vertex $a$ and a single loop at that vertex.  Then $\mathcal H(X^G)$ is 3-dimensional, spanned by $\varnothing$, $a$, and $G$.  There is no nontrivial multiplication, since there are no disjoint subgraphs.
\end{example}

If we consider 2-Segal sets associated to graphs, we get the following properties on their Hall algebras.

\begin{prop}\label{thm:graph-hall}
	Let $X^G$ be the 2-Segal set associated to a finite graph $G$.  The associated Hall algebra $\mathcal H^G=\mathcal H(X^G)$ is characterized as follows.
	\begin{enumerate}
		\item The basis element corresponding to $\varnothing$ is the multiplicative identity, so $H \cdotp \varnothing = H = \varnothing\cdotp H$ for every subgraph $H$ of $G$.
		
		\item If $H \cap K \neq \varnothing$, then $H \cdotp K = 0$.
		
		\item If $a$ and $b$ are distinct vertices of $G$, then 
		\[ a \cdotp b = \sum_H H \]
		where $H$ is a graph for which $v(H) = \{a, b\}$.  If there are $n$ edges between $a$ and $b$ in $G$, then there are $2^n$ such subgraphs $H$.
		
		\item More generally, if $H \cap K = \varnothing$ for two nonempty subgraphs $H$ and $K$ of $G$, then 
		\[ H \cdotp K = \sum_J J \] 
		where $J$ is a graph such that:
		\begin{itemize}
			\item the disjoint union $H \amalg K$ as a subgraph; and
			
			\item either $J= H \amalg K$ or the difference between $J$ and $H \amalg K$ consists of edges connecting vertices of $H$ with vertices of $K$.
		\end{itemize}
		
		\item \label{comm} It is commutative, so $H \cdotp K = K \cdotp H$ for any subgraphs $H$ and $K$ of $G$.
	\end{enumerate}
\end{prop}

Let us now consider the examples from trees.

\begin{example}
	Consider again the graph
	\[ \xymatrix@1{\bullet_a \ar@{-}[r] & \bullet_b} \]
	but now as a tree $T$ rooted at the vertex $a$. Now $\mathcal H(X^T)$ is 4-dimensional, spanned by $\varnothing$, $a$, $b$, and $G$, since $a \amalg b$ is not an admissible subforest.  The only nontrivial multiplication is 
	\[ a \cdotp b = G. \]
	Note that $b \cdotp a =0$ here, due to the directionality of the tree.
\end{example}

In comparison with Hall algebras from graphs, we have the following properties.

\begin{prop} 
	Let $T$ be a labeled tree and $X^T$ its associated 2-Segal set.  Then its Hall algebra $\mathcal H(X^T)$ has the following properties.
	\begin{itemize}
		\item If $H \cap K \neq \varnothing$, then $H \cdotp K = 0$.
	
		\item If $H \cap K = \varnothing$, then 
		\[ H \cdotp K = \sum_J J, \]
		where $J$ is tree with an admissible cut between $H$ and $K$.
	
		\item As before, the coefficients are always either 0 or 1.
	
		\item However, $\mathcal H(X^T)$ is not commutative unless $T$ is the trivial tree with a root and no edges.
	\end{itemize}
\end{prop}

Since there is an inclusion $X^T \hookrightarrow X^G$, where $G$ is the underlying graph of the tree $T$.  However, the Hall algebra is generally not functorial, so we do not get an induced map $\mathcal H(X^T) \rightarrow \mathcal H(X^G)$.  For such a map to exist, the inclusion map $X^T \hookrightarrow X^G$ would need to be both CULF and relatively Segal, which is not the case, unless $T$ is just a single vertex.  We refer the reader to \cite{trees} for these definitions and more details.

\section{Double categories and the discrete $S_\bullet$-construction}

Both papers \cite{dk} and \cite{gckt} have as a central example the 2-Segal space given by the Waldhausen $S_\bullet$-construction applied to an exact category.  Here, to keep to the setting of 2-Segal sets, we describe a discrete version that was developed in \cite{waldsets}.  Although it is too strict for most examples that arise naturally, it allows for very concrete descriptions of examples.

In \cite{waldsets}, our goal was to identify the most general input for the discrete $S_\bullet$-construction so that the output was a 2-Segal set.  In order to state this result, we need some further definitions.

\begin{definition}
	A \emph{double category} is a category internal to categories.
\end{definition}

In other words, a small double category $\mathcal{D}$ consists of a set of objects $\ob(\mathcal{D})$, a set of horizontal morphisms $\hor(\mathcal{D})$, a set of vertical morphisms $\ver(\mathcal{D})$, and a set of squares $\sq(\mathcal{D})$, and these sets are related by various source, target, identity and composition maps. In particular, $\Hor\mathcal{D} :=(\ob(\mathcal{D}), \hor(\mathcal{D))}$ and $\Ver\mathcal{D}:= (\ob(\mathcal{D}), \ver(\mathcal{D}))$ form categories. We typically denote horizontal morphisms by $\rightarrowtail$ and vertical morphisms by $\twoheadrightarrow$. 

\begin{definition}
	A double category $\mathcal D$ is \emph{pointed} if it is equipped with a distinguished object $\ast$ that is both initial in the horizontal category and terminal in the vertical category.  
\end{definition}

We would like to require squares in a double category to behave similarly to commutative squares that are both pushout and pullback squares, namely bicartesian squares.  Since the two kinds of morphisms in a double category need not both be morphisms in a common ambient category, we can no longer use the language of pushouts and pullbacks, but we can mimic their properties as follows.

\begin{definition}
	A double category is \emph{stable} if every square is uniquely determined by its span of source morphisms and, independently, by its cospan of target morphisms, so that the maps
	%.  More precisely, we require the maps
%	\[ \begin{aligned}
%		(s_h, s_v)\colon \Sq(\mathcal D) \to \Ver{\mathcal D} \times_{\ob(\mathcal D)} \Hor(\mathcal D) \\
%		(t_h, t_v)\colon \Sq(\mathcal D) \to \Ver(\mathcal D) \times_{\ob(\mathcal D)} \Hor(\mathcal D)
%	\end{aligned} \]
%	depicted by
	\begin{center}
		\begin{tikzpicture}[scale=0.6]
			\def\l{2cm}
			\begin{scope}
				\draw[fill] (0,0) circle (1pt) node (b0){};
				\draw[fill] (-\l,\l) circle (1pt) node (b2){};
				\draw[fill] (-\l,0) circle (1pt) node (b3){};
				\draw[fill] (0,\l) circle (1pt) node (b1){};
				
				\draw[epi] (b1)--node[anchor=west](x01){}(b0);
				\draw[mono] (b2)--node[anchor=south](x12){}(b1);
				\draw[epi] (b2)--node[anchor=east](x23){}(b3);
				\draw[mono] (b3)--node[anchor=north](x03){}(b0);
				
				\draw[twoarrowlonger] (-0.8*\l, 0.8*\l)--(-0.2*\l,0.2*\l);
			\end{scope}
			\draw[|->] (0.1*\l, 0.5*\l)--node[above](op1){}(0.4*\l,0.5*\l);
			\begin{scope}[xshift=1.5*\l]
				\draw[fill] (-\l,\l) circle (1pt) node (b2){};
				\draw[fill] (-\l,0) circle (1pt) node (b3){};
				\draw[fill] (0,\l) circle (1pt) node (b1){};
				
				\draw[mono] (b2)--node[anchor=south](x12){}(b1);
				\draw[epi] (b2)--node[anchor=east](x23){}(b3);
			\end{scope}
			\draw (2.5*\l, 0.5*\l) node{and};
			\begin{scope}[xshift=4.5*\l]
				\draw[fill] (0,0) circle (1pt) node (b0){};
				\draw[fill] (-\l,\l) circle (1pt) node (b2){};
				\draw[fill] (-\l,0) circle (1pt) node (b3){};
				\draw[fill] (0,\l) circle (1pt) node (b1){};
				
				\draw[epi] (b1)--node[anchor=west](x01){}(b0);
				\draw[mono] (b2)--node[anchor=south](x12){}(b1);
				\draw[epi] (b2)--node[anchor=east](x23){}(b3);
				\draw[mono] (b3)--node[anchor=north](x03){}(b0);
				
				\draw[twoarrowlonger] (-0.8*\l, 0.8*\l)--(-0.2*\l,0.2*\l);
			\end{scope}
			
			\draw[|->] (4.6*\l, 0.5*\l)--node[above](op1){}(4.9*\l,0.5*\l);
			\begin{scope}[xshift=6*\l]
				\draw[fill] (0,0) circle (1pt) node (b0){};
				\draw[fill] (-\l,0) circle (1pt) node (b3){};
				\draw[fill] (0,\l) circle (1pt) node (b1){};
				
				\draw[epi] (b1)--node[anchor=west](x01){}(b0);
				\draw[mono] (b3)--node[anchor=north](x03){}(b0);
			\end{scope}
		\end{tikzpicture}
	\end{center}
	are bijections. 
\end{definition}

\begin{example} \label{w2}
	Consider a double category, denoted by $\mathcal W_2$, with objects $ij$ for $0\leq i \leq j\leq 2$, and generated by the following non-identity horizontal and vertical morphisms and squares:
	\[ \begin{tikzpicture}   
		\draw (1,0.3) node(a00){$00$};
		\draw (2,0.3) node(a01){$01$};
		\draw (2,-0.7) node(a11) {$11$};
		\draw (3, -0.7) node(a12){$12$};
		\draw (3, 0.3) node(a02){$02$};
		\draw (3,-1.7) node (a22){$22.$};
		
		\draw[mono] (a00)--(a01);
		\draw[epi] (a01)--(a11);
		\draw[mono] (a11)--(a12);
		\draw[mono] (a01)--(a02);
		\draw[epi] (a12)--(a22);
		\draw[epi] (a02)--(a12);
		
		\draw[twoarrowlonger] (2.2,0.1)--(2.8,-0.5);
	\end{tikzpicture} \]
	Since the only non-identity square is uniquely determined by its pair of sources or by its pair of targets, $\mathcal W_2$ is a stable double category.  Furthermore, we assume that each $ii$ is a given zero object, so that $\mathcal W_2$ is pointed.
\end{example}

The pointed stable double categories are our desired input for the discrete $S_\bullet$-construction.  We denote by $\sdc$ the category of such.

\begin{definition}
	Let $\mathcal C$ be a pointed stable double category.  Let $\mathcal S_n(\mathcal C)$ be the set of diagrams in $\mathcal C$ of the form
	\[ \xymatrix{\ast \ar@{>->}[r] & a_{01} \ar@{>->}[r] \ar@{->>}[d] & a_{02} \ar@{>->}[r] \ar@{->>}[d] & \cdots \ar@{>->}[r] & a_{0n} \ar@{->>}[d] \\
		& \ast \ar@{>->}[r] & a_{12} \ar@{>->}[r] \ar@{->>}[d] & \cdots \ar@{>->}[r] & a_{1n} \ar@{->>}[d] \\
		&& \ast & \ddots & \vdots \ar@{->>}[d] \\
		&&&& \ast. } \].
	Let $\mathcal S_\bullet(\mathcal C)$ be the simplicial set comprised of the sets $\mathcal S_n(\mathcal C)$ for all $n \geq 0$.
\end{definition}

Alternatively, we could define $\mathcal S_n(\mathcal C)$ in terms of suitable functors from generalizations of the double category in Example \ref{w2}, the approach that is taken in \cite{waldsets}.  

\begin{theorem} \cite[4.8]{waldsets}
	If $\mathcal D$ is a pointed stable double category, then the simplicial set $S_\bullet(\mathcal D)$ is a 2-Segal set.
\end{theorem}

In fact, this 2-Segal set is \emph{reduced}, in the sense that its set of 0-simplices has a single element.  We denote the category of reduced Segal sets by $2\Seg_*$.

\begin{proof}[Idea of the proof]
	Let us consider the 2-Segal map
	\[ (d_0, d_2) \colon \mathcal S_3(\mathcal C) \rightarrow \mathcal S_2(\mathcal C) \times_{\mathcal S_1(\mathcal C)} \mathcal S_2(\mathcal C). \]
	Given a 3-simplex
	\[ \xymatrix{\ast \ar@{>->}@[red][r] & a_{01} \ar@{>->}[r] \ar@{->>}@[red][d] \ar@/^1pc/@[red][rr] & a_{02} \ar@{>->}[r] \ar@{->>}[d] &  a_{03} \ar@{->>}@[red][d] \\
		& \ast \ar@{>->}@[blue][r] \ar@/^1pc/@[red][rr] & a_{12} \ar@{>->}@[blue][r] \ar@{->>}@[blue][d]& a_{13} \ar@{->>}@[blue][d] \ar@/^1pc/@[red][dd] \\
		&& \ast \ar@{>->}@[blue][r]  & a_{23} \ar@{->>}@[blue][d] \\
		&&& \ast } \]
	the part of the diagram depicted in blue is the image of the map $d_0$, which removes the top row.  The image of the map $d_2$ is the 2-simplex depicted in red.  To recover the whole 3-simplex from the red and blue portions, it suffices to fill in the object $a_{02}$ and the arrows into and out of it.  We can use stability to identify the object $a_{02}$ filling out the top right-hand square; the analogue to the universal property for pullbacks in this setting provides the map $a_{01} \rightarrowtail a_{02}$.
\end{proof}

We now consider a functor in the reverse direction, taking a reduced 2-Segal set to a pointed stable double category.  We begin by recalling the following endofunctors of $\Delta$:
\[ \begin{aligned}
	i \colon \Delta  \to & \Delta\\
	[n]  \mapsto & [0] \ast [n] 
\end{aligned}
\qquad \text{and} \qquad
\begin{aligned}
	f \colon \Delta  \to & \Delta\\
	[n] \mapsto & [n] \ast [0],
\end{aligned} \]
where $\ast$ denotes the join operation, which here is concatentation of linearly ordered sets.  The functor names are meant to suggest adjoining an initial and final object, respectively, and in particular, both $[0] \ast [n]$ and $[n] \ast [0]$ are precisely the object $[n+1]$.  More specifically, there are natural transformations $i \Rightarrow \id_{\Delta}$ and $f\Rightarrow \id_{\Delta}$ induced by the maps $d^0\colon [n] \to [0] \ast [n]=[n+1]$ and $d^{n+1}\colon [n] \to [n] \ast [0]=[n+1]$, respectively.

\begin{definition}
	Given a simplicial set $X$, its \emph{path spaces} are the simplicial sets $P^{\triangleleft}(X)=X\circ i^{\op}$ and $P^\triangleright(X)=X\circ f^{\op}$. 
\end{definition}

The natural transformations above induce maps of simplicial sets 
\[ P^\triangleleft (X) \to X \quad \text{and} \quad  \colon P^\triangleright(X)\to X \]
that are natural in $X$. These simplicial sets are sometimes called \emph{d\'ecalages}, for example in \cite{gckt}.

The following Path Space Criterion from \cite{dk} and \cite{gckt} relates $1$-Segal sets and $2$-Segal sets.
%[Theorem 6.3.2]

\begin{theorem}\label{thm:PathSpaceCriterion}
	A simplicial set $X$ is $2$-Segal if and only if its path spaces $P^{\triangleleft} X$ and $P^\triangleright X$ are $1$-Segal sets.
\end{theorem}

The idea behind the inverse functor to $\mathcal S_\bullet$ is to build a double category from both path spaces.  If $X$ is a 2-Segal set, then by the above theorem, the path spaces $P^\triangleleft X$ and $P^\triangleright X$ are 1-Segal, and hence correspond to categories that we denote suggestively by $\mathcal H$ and $\mathcal V$.  Both have $X_1$ as the set of objects and $X_2$ as the set morphisms, and our aim is to assemble them to a double category with horizontal category $\mathcal H$ and vertical category $\mathcal V$. The source, target, and identity maps for $\mathcal H$ are given by $d_2$, $d_1$, and $s_1$, respectively, and for $\mathcal V$, they are given by $d_1$, $d_0$, and $s_0$. Composition  in $\mathcal H$ and $\mathcal V$ can be defined by
\[
X_2\times_{X_1} X_2 \underset{\cong}{\xrightarrow{(d_3,d_1)^{-1}}}  X_3 \xrightarrow{d_2}   X_2 \text{ and } X_2\times_{X_1} X_2 \underset{\cong}{\xrightarrow{(d_2,d_0)^{-1}}} X_3 \xrightarrow{d_1}   X_2,
\]
respectively. 

Let $\Sq = X_3$. We define the horizontal source of a square by the face map $d_3$ and its horizontal target by $d_2$, as follows:
\[ \xymatrix{\mor \mathcal V\ar@{=}[d]&\Sq \ar@{-->}[l]_-{s_h}\ar@{-->}[r]^-{t_h}\ar@{=}[d]&\mor \mathcal V \ar@{=}[d]\\X_2&X_3\ar[l]^-{ d_3}\ar[r]_-{ d_2}&X_2.} \]
We can similarly define the vertical source and target of a square using the face maps $d_1$ and $d_0$, respectively:
\[ \xymatrix{\mor \mathcal H \ar@{=}[d]&\Sq \ar@{-->}[l]_-{s_v}\ar@{-->}[r]^-{t_v}\ar@{=}[d]&\mor \mathcal H \ar@{=}[d]\\X_2&X_3\ar[l]^-{ d_1}\ar[r]_-{ d_0}&X_2.} \]
More explicitly, for any $x\in X_3$ whose 2-dimensional faces are $m,m', e,$ and $e'$ as depicted in the diagram
\begin{center}
	\begin{tikzpicture}[scale=0.7]
		\def\l{3cm}
		\begin{scope} 
			\draw[fill] (-\l,0) node (b0){$0$};
			\draw[fill] (0,0)  node (b1){$1$};
			\draw[fill] (0,\l) node (b2){$2$};
			\draw[fill] (-\l,\l)  node (b3){$3$};
			
			\draw (b1)--(b2);
			\draw (b2)--(b3);
			\draw (b3)--(b0);
			\draw (b1)--(b0);
			\draw (b1)--(b3);
			
			\draw[epi] (-0.95*\l, 0.45*\l)-- node[below]{$e'$}(-0.55*\l, 0.45*\l);
			\draw[mono] (-0.05*\l, 0.55*\l)--node[above]{$m'$}(-0.45*\l, 0.55*\l);
		\end{scope}
		
		\begin{scope} [xshift=4cm]
			\draw[fill] (-\l,0) node (b0){$0$};
			\draw[fill] (0,0)  node (b1){$1,$};
			\draw[fill] (0,\l) node (b2){$2$};
			\draw[fill] (-\l,\l)  node (b3){$3$};
			
			\draw (b1)--(b2);
			\draw (b2)--(b3);
			\draw (b3)--(b0);
			\draw (b1)--(b0);
			\draw (b0)--(b2);
			
			\draw[epi] (-0.45*\l, 0.45*\l)-- node[below]{$e$}(-0.05*\l, 0.45*\l);
			\draw[mono] (-0.55*\l, 0.55*\l)--node[above]{$m$}(-0.95*\l, 0.55*\l);
		\end{scope}
		
	\end{tikzpicture}
\end{center}
the corresponding square in $\Sq$ is given by
\begin{center}
	\begin{tikzpicture}
		\def\l{3cm}
		\def\h{3cm}
		%nodes
		\draw (0,0) node (a1){$d_2d_1(x)$};
		\draw (\l,0) node (a2){$d_1d_1(x)$};
		\draw(-0.3*\l,-0.3*\h) node (b1){$d_1d_3(x)$};
		\draw(1.3*\l,-0.3*\h) node (b2){$d_1d_2(x)$};
		\draw(-0.3*\l,-1*\h) node (c1){$d_0d_3(x)$};
		\draw(1.3*\l,-1*\h) node (c2){$d_0d_2(x)$};
		\draw (0,-1.3*\h) node (d1){$d_2d_0(x)$};
		\draw (\l,-1.3*\h) node (d2){$d_1d_0(x).$};
		%middle
		\draw (0.6*\l, -0.6*\h)node(middle){$x$};
		\draw[twoarrowlonger] (0.3*\l, -0.4*\h)--(0.7*\l, -0.9*\h);
		%arrows
		\draw[mono] (a1)--node[above](u){$m=d_1(x)$} (a2);
		\draw[epi] (b1)--node[left](l){$e=d_3(x)$} (c1);
		\draw[epi] (b2)--node[right](r){$e'=d_2(x)$} (c2);
		\draw[mono] (d1)--node[below](b){$m'=d_0(x)$} (d2);
		%equalities
		\draw[double, double distance=1pt] (a1)--(b1);
		\draw[double, double distance=1pt] (a2)--(b2);
		\draw[double, double distance=1pt] (c1)--(d1);
		\draw[double, double distance=1pt] (c2)--(d2);
	\end{tikzpicture}
\end{center}

\begin{prop} \cite[5.2]{waldsets}
	Let $X$ be a 2-Segal set. Then the above construction defines a stable double category $\mathcal PX$. If $X$ is reduced, then $\mathcal PX$ is pointed.
\end{prop}

\begin{theorem} \cite[6.1]{waldsets}
	The functors $S_\bullet$ and $\mathcal P$ define an equivalence of categories 
	\[ S_\bullet \colon \sdc_* \overset{\simeq}\leftrightarrow 2\Seg_* \colon \mathcal P. \]
\end{theorem}

In fact, in that paper we prove a more general of this theorem, where we show that the category of all 2-Segal sets is equivalent to the category of \emph{augmented} stable double Segal spaces.  For simplicity, we do not discuss the augmented case here, but refer the interested reader to \cite{waldsets}.

Let us now look at some of our examples of 2-Segal sets from earlier in the paper and describe their corresponding pointed stable double categories.

\begin{example}
	Recall the 2-Segal set $M_\bullet$ given by the nerve of a partial monoid.  Using the arguments above, the objects of the associated double category $\mathcal P M$ are the elements of the monoid $a\in M$. The sets of horizontal and vertical morphisms are both given by the set of composable pairs $M_2$, but structure is different.  For $(a,b) \in M_2$, the corresponding horizontal arrow has source $a$ and target $a\cdot b$,  so can be interpreted as multiplication on the right by $b$:
	\[ \begin{tikzcd}
		a \arrow[tail]{r}{\cdot b} & a\cdot b.
	\end{tikzcd} \]
	On the other hand, the vertical arrow corresponding to $(a,b)\in M_2$ has target $b$ and source $a\cdot b$, so can be understood as multiplication on the left by $a$, but with the arrow pointing in the other direction:
	\[ \begin{tikzcd}
		a\cdot b \arrow[two heads, swap]{d}{a \cdot} \\  b.
	\end{tikzcd} \]
	The set of squares is given by $M_3$ and reflects associativity of the multiplication: for $(a,b,c)\in M_3$, its associated square is
	\[ \begin{tikzcd}
		a\cdot b \arrow[two heads, swap]{d}{a \cdot} \arrow[tail]{r}{\cdot c} & (a\cdot b) \cdot c = a\cdot (b\cdot c) \arrow[two heads]{d}{a \cdot}  \\
		b \arrow[tail, swap]{r}{\cdot c} & b\cdot c.
	\end{tikzcd} \]
	The double category $\mathcal M$ is pointed by the unit $1\in M$, since for every $a\in M$, the elements $(1,a)\in M_2$ and $(a,1)\in M_2$, which can be visualized as
	\[ \begin{tikzcd}
		1 \arrow[tail]{r}{\cdot a} & 1\cdot a = a
	\end{tikzcd}
	\hspace{1cm}\mbox{and}\hspace{1cm}
	\begin{tikzcd}
		a=a\cdot 1 \arrow[two heads, swap]{d}{a \cdot} \\  1,
	\end{tikzcd} \]
	exhibit 1 as initial with respect to the horizontal category and terminal with respect to the vertical category.  Finally, stability is again given by associativity, since both
	\[ \begin{tikzcd}
		a\cdot b \arrow[two heads, swap]{d}{a \cdot} \arrow[tail]{r}{\cdot c} & (a\cdot b) \cdot c   \\
		b 
	\end{tikzcd}
	\quad\mbox{and}\quad
	\begin{tikzcd}
		\mbox{} & a\cdot (b\cdot c) \arrow[two heads]{d}{a \cdot}  \\
		b \arrow[tail, swap]{r}{\cdot c} & b\cdot c
	\end{tikzcd} \]
	can be completed uniquely to a square as above.
\end{example}

\begin{example} 
	We apply our construction to the 2-Segal set $X$ associated to a graph $G$.
	
	An element in $X_2$, such as 
	\[\begin{tikzpicture}
		\def\l{1cm}
		\draw (\l,0) node [vnode=$a$](a){};
		\draw (2*\l,0) node [vnode=$b$](b){};
		\draw (\l,-0.3*\l) node[vellipsefirstone=0.8*\l]{};
		\draw (2*\l, -0.3*\l) node[vellipsesecondone=0.8*\l]{};
		\draw (a)--(b);
	\end{tikzpicture} \]
	represents a horizontal morphism
	\begin{center}
		\begin{tikzpicture}
			\def\l{1cm}
			\draw (\l,0) node [vnode=$a$](a){};
			\draw (\l,-0.3*\l) node[vellipsefirstone=0.8*\l]{};
			
			\draw[mono] (1.5*\l,-0.3*\l) -- (5.5*\l,-0.3*\l);
			
			%label of the mono
			\begin{scope}[yshift=0.8*\l, xshift=3*\l]
				\draw (0,0) node [vnode=$a$](a){};
				\draw (\l,0) node [vnode=$b$](b){};
				\draw (0,-0.3*\l) node[vellipsefirstone=0.8*\l]{};
				\draw (\l, -0.3*\l) node[vellipsesecondone=0.8*\l]{};
				\draw (a)--(b);
			\end{scope}
			
			\begin{scope}[xshift=5*\l]
				\draw (\l,0) node [vnode=$a$](a){};
				\draw (2*\l,0) node [vnode=$b.$](b){};
				\draw (a)--(b);
				\draw (1.5*\l, -0.3*\l) node[vellipsefirsttwo=0.8*\l]{};
			\end{scope}
			
		\end{tikzpicture}
	\end{center}
	
	The same element gives the vertical morphism
	\begin{center}
		\begin{tikzpicture}[scale=0.75]
			\def\l{1cm}
			\def\w{6cm}
			\def\h{-5cm}
			
			\begin{scope}
				\draw (\l,0) node [vnode=$a$](a){};
				\draw (2*\l,0) node [vnode=$b$](b){} ++(0.5*\l,-0.5*\l) node (t1) {};
				\draw (1.5*\l,-0.3*\l) node[vellipsefirsttwo=0.8*\l]{};
				\draw (a)-- node (l11) {} (b);
				\path (l11) ++(0,-\l) node (l1) {};
			\end{scope}
			
			\begin{scope}[shift={(0.5, \h)}]
				\draw (\l,0) node [vnode=$b.$](b){} node[anchor=south, yshift=0.2cm] (l2) {} ++(0.5*\l,-0.5*\l) node (b1) {};
				\draw (\l,-0.3*\l) node[vellipsefirstone=0.8*\l]{};
			\end{scope}
			
			\draw[epi] (l1) -- (l2);
			%label for leftmost arrow
			\begin{scope}[shift={(-0.2*\l, 0.5*\h)}]
				\draw (0,0) node [vnode=$a$](a){};
				\draw (\l,0) node [vnode=$b$](b){};
				\draw (0,-0.3*\l) node[vellipsefirstone=0.8*\l]{};
				\draw (\l, -0.3*\l) node[vellipsesecondone=0.8*\l]{};
				\draw (a)--(b);
			\end{scope}
		\end{tikzpicture}
	\end{center}
	
	The 3-simplex
	\[ \begin{tikzpicture}
		\def\l{1cm}
		\draw (\l,0) node [vnode=$a$](a){};
		\draw (2*\l,0) node [vnode=$b$](b){};
		\draw (3*\l,0) node [vnode=$c$](c){};
		\draw (\l,-0.3*\l) node[vellipsefirstone=0.8*\l]{};
		\draw(2*\l, -0.3*\l) node[vellipsesecondone=0.8*\l]{};
		\draw (3*\l, -0.3*\l) node[vellipsethirdone=0.8*\l]{};
		\draw (a)--(b);
		\draw (b)--(c);
	\end{tikzpicture} \]
	gives rise to a square
	\begin{center}
		\begin{tikzpicture}[scale=0.75]
			\def\l{1cm}
			\def\w{6cm}
			\def\h{-5cm}
			
			\begin{scope}
				\draw (\l,0) node [vnode=$a$](a){};
				\draw (2*\l,0) node [vnode=$b$](b){} ++(0.5*\l,-0.5*\l) node (t1) {};
				\draw (1.5*\l,-0.3*\l) node[vellipsefirsttwo=0.8*\l]{};
				\draw (a)-- node (l11) {} (b);
				\path (l11) ++(0,-\l) node (l1) {};
			\end{scope}
			
			\begin{scope}[xshift = \w]
				\draw (\l,0) node [vnode=$a$](a){} ++(-0.8*\l,-0.5*\l) node (t2) {};
				\draw (2*\l,0) node [vnode=$b$](b){} ++(0,-\l) node (r1) {};
				\draw (3*\l,0) node [vnode=$c$](c){};
				\draw (2*\l, -0.3*\l) node[vellipsefirstthree=0.8*\l]{};
				\draw (a)--(b);
				\draw (b)--(c);
			\end{scope}
			
			\draw[mono] (t1) -- node (T) {} (t2);
			%label of the top mono 
			\begin{scope}[shift={(0.5*\w++0.5*\l, 1.2*\l)}]
				\draw (0,0) node [vnode=$a$](a){};
				\draw (\l,0) node [vnode=$b$](b){};
				\draw (2*\l,0) node [vnode=$c$](c){};
				\draw (b)--(c);
				\draw (a)--(b);
				\draw (0.5*\l,-0.3*\l) node[vellipsefirsttwo=0.8*\l]{};
				\draw (2*\l, -0.3*\l) node[vellipsesecondone=0.8*\l]{};
			\end{scope}
			
			\begin{scope}[shift={(0.5, \h)}]
				\draw (\l,0) node [vnode=$b$](b){} node[anchor=south, yshift=0.2cm] (l2) {} ++(0.5*\l,-0.5*\l) node (b1) {};
				\draw (\l,-0.3*\l) node[vellipsefirstone=0.8*\l]{};
			\end{scope}
			
			\draw[epi] (l1) -- (l2);
			%label for leftmost arrow
			\begin{scope}[shift={(-0.2*\l, 0.5*\h)}]
				\draw (0,0) node [vnode=$a$](a){};
				\draw (\l,0) node [vnode=$b$](b){};
				\draw (0,-0.3*\l) node[vellipsefirstone=0.8*\l]{};
				\draw (\l, -0.3*\l) node[vellipsesecondone=0.8*\l]{};
				\draw (a)--(b);
			\end{scope}
			
			\begin{scope}[shift={(\w ++ 0.5cm, \h)}]
				\def\l{1cm}
				\draw (\l,0) node [vnode=$b$](b){} ++(-0.5*\l,-0.5*\l) node (b2) {};
				\draw (2*\l,0) node [vnode=$c.$](c){};
				\draw (b)-- node[anchor=south, yshift=0.2cm] (r2) {} (c);
				\draw (1.5*\l,-0.3*\l) node[vellipsefirsttwo=0.8*\l]{};
			\end{scope}
			
			\draw[epi] (r1) -- node (R) {} (r2);
			%label for the rightmost arrow
			\path (R) ++ (0.5cm, 0) node (R1) {};
			\begin{scope}[shift={(R1)}]
				\draw (0,0) node [vnode=$a$](a){};
				\draw (\l,0) node [vnode=$b$](b){};
				\draw (2*\l,0) node [vnode=$c$](c){};
				\draw (b)--(c);
				\draw (a)--(b);
				\draw (0,-0.3*\l) node[vellipsefirstone=0.8*\l]{};
				\draw (1.5*\l, -0.3*\l) node[vellipsesecondtwo=0.8*\l]{};
			\end{scope}
			
			\draw[mono] (b1) -- node (B) {} (b2);
			\path (B) ++ (-0.5, -1) node (B1) {};
			\begin{scope}[shift={(B1)}]
				\draw (0,0) node [vnode=$b$](a){};
				\draw (\l,0) node [vnode=$c$](b){};
				\draw (0,-0.3*\l) node[vellipsefirstone=0.8*\l]{};
				\draw (\l, -0.3*\l) node[vellipsesecondone=0.8*\l]{};
				\draw (a)--(b);
			\end{scope}
			
		\end{tikzpicture}
	\end{center}

\end{example}

We encourage the reader to develop the pointed stable double category associated $X^T$, where $X$ is a tree.  Several specific examples are given in \cite{trees}.

\end{document}